\documentclass[a4paper,11pt]{article}

\usepackage{amsmath}
\usepackage{amssymb}
\usepackage{amsthm}
\usepackage{graphicx}
\usepackage{enumerate}
\usepackage{authblk}
\usepackage{bm}
\usepackage{hyperref}

\usepackage{xcolor}

\newtheorem{theorem}{Theorem}[section]
\newtheorem{proposition}[theorem]{Proposition}

\newtheorem{conjecture}[theorem]{Conjecture}

\theoremstyle{definition}
\newtheorem{definition}[theorem]{Definition}

\begin{document}
	
\title{Markov's equation is not partition regular}
\author[1]{Tianyi Tao\thanks{Email: tytao20@fudan.edu.cn}}
\author[2]{Bohan Yang\thanks{Email: bhyang@simis.cn}}
\affil[1]{School of Mathematical Sciences\\Fudan University\\Shanghai 200433, P. R. China}
\affil[2]{Shanghai Institute for Mathematics and Interdisciplinary Sciences (SIMIS)\\Shanghai 200433, P. R. China}

\date{\today}
\maketitle

\begin{abstract}
Markov's equation \(x^2 + y^2 + z^2 = 3xyz\) is a widely studied topic in number theory, and the structure of its solutions has profound connections with mathematical fields such as combinatorics, hyperbolic geometry, approximation theory, and cluster algebras. In this paper, we prove that Markov's equation is not partition regular, which also confirms a necessary condition for the Uniqueness Conjecture.
\end{abstract}


\section{Introduction}

\begin{definition}\label{pr}
   A Diophantine equation \(f(x_1, \cdots, x_n) = 0\) is said to be \textit{partition regular} over \(\mathbb{N}\) if, for any finite partition of \(\mathbb{N}\) into \(\mathbb{N} = \bigcup\limits_{i=1}^m A_i\), there exists a cell \(A_i\) such that \(A_i\) contains infinitely many solutions to the equation. 
\end{definition}

Research on partition regularity problems originated from Schur's theorem \cite{schur1916}, which proved that equation \(x + y = z\) is partition regular. Later, Rado \cite{rado} gave the necessary and sufficient conditions for the partition regularity of homogeneous linear systems \(\bm{A}\bm{x} = \bm{0}\). However, the problem of partition regularity for inhomogeneous equations is much more difficult, among which the most studied one is presumably the partition regularity of the Pythagorean equation \(x^2 + y^2 = z^2\) \cite{graham2007some}, which remains an open problem to this day; recent progress on it can be found in \cite{frantzikinakis2017higher,frantzikinakis2025partition}. In addition, the method used by Bergelson in \cite{bergelson1996polynomial} can be widely applied to handle partition regularity problems where a linear term of some parameter in a linear equation is replaced by a polynomial term. Di Nasso and Luperi Baglini discussed the partition regularity of many inhomogeneous equations in \cite{di2018ramsey}.

In contrast, many works negate the partition regularity of certain specific equations. For example, Csikvári et al.\cite{csikvari2012density} proved that the equation \(x + y = z^2\) is not partition-regular; Di Nasso and Riggio \cite{di2018fermat} proved that a wide class of Fermat-like equations are not partition regular, such as \(x^3 + y^4 = z^5\). Note that in \cite{di2018ramsey}, there are several methods for determining that a Diophantine equation is not partition-regular, but these methods all fail to apply to the Markov equation of interest in this paper.

The Markov equation \(x^2 + y^2 + z^2 = 3xyz\) is a pivotal object in number theory. This equation is known to form deep connections that span diverse fields, including the Diophantine approximation, hyperbolic geometry, dynamical systems, and combinatorics, as discussed in \cite{CF89,DCCS21}.

The recursive structure of its positive integer solutions (Markov triples) is linked to binary indefinite quadratic forms, a connection first identified by Markov \cite{Markov1880}. Later, Perron \cite{Perron21} further explored this link using continued fractions. Geometrically, Cohn \cite{Cohn71,Cohn78} and Haas \cite{Haas} established a correspondence between Markov triples and geodesics on certain surfaces such as \(SL(2,\mathbb{R})/SL(2,\mathbb{Z})\). These works play an important role in the study of closed geodesics on Riemannian surfaces.
In combinatorial group theory, Cohn \cite{Cohn72,Cohn93} analyzed these structures using words and trace algebra \cite{Fricke1896}, while all Markov triples can be represented as vertices of an infinite binary tree \cite{Cohn79}. Together, these cross-field links solidify the role of the equation as a cross-disciplinary hub in mathematics.

\begin{definition}
A \textit{Markov triple} is a triple of positive integers \((x, y, z)\) satisfying the Markov equation \(x^2 + y^2 + z^2 = 3xyz\). A \textit{Markov number} is a positive integer appearing in some Markov triple; we denote the set of all Markov numbers by \(\mathcal{M}\).
\end{definition}

\begin{proposition}[Corollary 3.5 of \cite{aigner2013markov}]
Every Markov number is the maximum element in at least one Markov triple.
\end{proposition}

In 1913, Frobenius \cite{Frobenius13} proposed an important conjecture that two distinct Markov triples cannot share the same maximum element. This conjecture has remained open for over a century and holds significant importance in the field.

\begin{conjecture}[Uniqueness Conjecture]
Every Markov number is the maximum element in exactly one Markov triple.
\end{conjecture}

The Uniqueness Conjecture is critical, as its validity would imply rigidity in multiple domains: it would enforce the uniqueness of certain binary quadratic forms and the unique equivalent class of simple closed geodesics with the same length in certain surfaces.

Despite extensive computer verification up to $10^{140}$ \cite{ Baragar96,Borosh75}, and moreover, proofs for Markov numbers of specific forms \cite{Baragar96, button1998uniqueness, button2001markoff, CC13}, the conjecture is nevertheless unresolved in its full generality. For more background and recent progress on the Uniqueness Conjecture and Markov's equation, readers are referred to \cite{aigner2013markov}.

\begin{proposition}
If the Uniqueness Conjecture holds, then Markov's equation is not partition regular over \(\mathbb{N}\).
\end{proposition}

\begin{proof}
From Chapter 3 of \cite{aigner2013markov}, all Markov triples form a tree. When the Uniqueness Conjecture holds, each Markov number \(m\) appears as the largest element in exactly one Markov triple; let \(d(m)\) denote the depth of this triple in the tree (the distance from the root node \((1,1,1)\)).

Let \(A = \{ m \in \mathcal{M} : d(m) \equiv 0 \pmod{2} \}\) and \(B = \mathbb{N} \setminus A\). Then Markov's equation has no non-trivial solutions in either \(A\) or \(B\).

\end{proof}

\section{The proof}

In this section, we prove that Markov's equation is not partition regular without relying on the Uniqueness Conjecture, which is inspired by \cite{csikvari2012density}, \cite{di2018fermat}, and the following proposition. Notably, our proof combines the traditional modular method with the structural properties of Markov's equation itself.

\begin{proposition}[Corollary 3.4 of \cite{aigner2013markov}]\label{coprime}
 If $(x,y,z)$ is a Markov triple, then $\gcd(x,y,z)=1$.   
\end{proposition}

\begin{theorem}
    There is a 9-coloring of $\mathbb N$ such that $x^2+y^2+z^2=3xyz$ has no non-trivial homochromatic solution.
\end{theorem}
\begin{proof}
We partition \(\mathbb{N}\) into 9 disjoint sets as follows:
\begin{itemize}
    \item \(A_i = \{ n \in \mathbb{N} : n \equiv i \pmod{5} \}\), \(i = 0, 2, 3, 4\);
    \item  \(B_i = \{ n \in \mathbb{N} : n = 5^{v(n)}m(n) + 1,\ 5 \nmid m(n),\ m(n) \equiv i \pmod{5} \}\), \(i = 1, 2, 3, 4\);
    \item \(C=\{1\}\).
\end{itemize}
Suppose \(x, y, z\) are nontrivial solutions to Markov's equation. By Proposition \ref{coprime}, \(x, y, z\) cannot all belong to \(A_0\); on the other hand, if \(x, y, z \in A_i\) for \(i \in \{2, 3, 4\}\), then \(x^2 + y^2 + z^2 \equiv 3i^2 \pmod{5}\) and \(3xyz \equiv 3i^3 \pmod{5}\), which leads to a contradiction.

In the following, suppose there exists \(i \in \{1, 2, 3, 4\}\) such that \(x, y, z \in B_i\). Let \(x = 5^{v(x)}m(x) + 1\), \(y = 5^{v(y)}m(y) + 1\), \(z = 5^{v(z)}m(z) + 1\). Denote \(P = x^2 + y^2 + z^2\) and \(Q = 3xyz\). Then  
\begin{itemize}
    \item \(
    \begin{aligned}[t]
    P &= 5^{2v(x)}m(x)^2 + 5^{2v(y)}m(y)^2 + 5^{2v(z)}m(z)^2 \\
    &+ 2 \cdot 5^{v(x)}m(x) + 2 \cdot 5^{v(y)}m(y) + 2 \cdot 5^{v(z)}m(z) + 3;
    \end{aligned}
    \)
    \item \(
    \begin{aligned}[t]
    Q &= 3 [ 5^{v(x)+v(y)+v(z)}m(x)m(y)m(z) \\
    &+5^{v(x)+v(y)}m(x)m(y) + 5^{v(x)+v(z)}m(x)m(z) + 5^{v(y)+v(z)}m(y)m(z) \\
    &+ 5^{v(x)}m(x) +  5^{v(y)}m(y) +  5^{v(z)}m(z) +1].
    \end{aligned}
    \)
\end{itemize}

Without loss of generality, assume \(v(x) \ge v(y) \ge v(z)=v\). Since 5 is a prime number and \(m(x), m(y), m(z)\) are congruent modulo 5, it follows that the 5-adic valuations of \(P - 3\) and \(Q - 3\) are both \(v\). Let \(P - 3 = 5^v \lambda\) and \(Q - 3 = 5^v \mu\); then

{\bf Case 1} \(v(x) = v(y) = v(z) = v.\)  

\(\lambda \equiv 2 m(x) + 2 m(y) + 2 m(z) \equiv  6i \equiv i \pmod{5}\), 

\(\mu \equiv 3 (m(x) + m(y) + m(z)) \equiv  9i \equiv 4i \pmod{5}\).  

{\bf Case 2} \(v(x) > v(y) = v(z) = v.\)  

\(\lambda \equiv 2 m(y) + 2 m(z) \equiv 4i \pmod{5}\),  

\(\mu \equiv 3 (m(y) + m(z)) \equiv 6i \equiv i \pmod{5}\).  

{\bf Case 3} \(v(x) \ge v(y) > v(z) = v.\)  

\(\lambda \equiv 2 m(z) \equiv 2i \pmod{5}\),  

\(\mu \equiv 3 m(z) \equiv 3i \pmod{5}\).

Combining the above cases, \(P\) and \(Q\) cannot be equal, which leads to a contradiction.
\end{proof}

\bibliographystyle{plain}
\bibliography{main}

\end{document}